\documentclass[12pt]{amsart}
\usepackage{amsfonts, amsbsy, amsmath, amssymb}

\usepackage[T1]{fontenc}
\usepackage[utf8]{inputenc}
\usepackage{lmodern}
\usepackage{amsmath, amsthm, amssymb,amscd, mathrsfs, amsfonts, mathtools,tikz-cd}
\usepackage{relsize}
\usepackage{euler}
\usepackage{times}
\usepackage[all]{xy}
\usepackage{todonotes}
\usepackage{xcolor}
\usepackage[bookmarks=true,
            bookmarksnumbered=true, breaklinks=true,
            pdfstartview=FitH, hyperfigures=false,
            plainpages=false, naturalnames=true,
            colorlinks=true,pagebackref=true,
            pdfpagelabels]{hyperref}
\newtheorem{theorem}{Theorem}[section]

\newtheorem{corollary}[theorem]{Corollary} 

\newtheorem{question}[theorem]{Question}

\usepackage[lite]{amsrefs}

\renewcommand{\PrintDOI}[1]{\href{http://dx.doi.org/\detokenize{#1}}{doi: \detokenize{#1}}%
  \IfEmptyBibField{pages}{, (to appear in print)}{}}

\def\commutatif{\ar@{}[rd]|{\circlearrowleft}}

\newtheorem{thm}{Theorem}[section]
\newtheorem{proposition}[thm]{Proposition}
\newtheorem{lemma}[thm]{Lemma}

\theoremstyle{definition}

\theoremstyle{remark}

\newtheorem{example}[thm]{Example}

\allowdisplaybreaks

\usepackage{bbm}

\title{ Circular orderability and quandles }

\author[Idrissa Ba]{Idrissa Ba}
\address{Department of Mathematics\\
Marianopolis College\\
Westmount\\
QC H3Y 1X9} 
\email{ba162006@yahoo.fr}
\author[Mohamed Elhamdadi]{Mohamed Elhamdadi} 
\address{Department of Mathematics, 
University of South Florida, Tampa, FL 33620, U.S.A.} 
\email{emohamed@math.usf.edu} 
\urladdr{ http://shell.cas.usf.edu/~emohamed/}

\begin{document}


\begin{abstract}
 In this paper, we introduce the notion of circular orderability for quandles. We show that the set all right (respectively left) circular orderings of a quandle is a compact topological space. We also show that the space of right (respectively left) orderings of a quandle embeds in its space of right (respectively left) circular orderings.  Examples of quandles that are not left circularly orderable and examples of quandles that are neither left nor right circularly orderable are given.
\end{abstract}

\maketitle 
\section{Introduction}
Quandles are algebraic structures whose axioms are motivated by the three Reidemeister moves in classical knot theory.  They were introduced independently by Joyce \cite{Joyce} and Matveev \cite{Matveev} in the early 1980's.  Since then they have been used to construct invariants of knots in the $3$-space and knotted surfaces in $4$-space (see for example \cite{EN} for more on quandles).  However, the notion of a quandle can be traced back to the 1940's in the work of Mituhisa Takasaki \cite{Takasaki}.  Joyce and Matveev introduced the notion the fundamental quandle of a knot leading to their main theorem which states that  two knots $K$ and $K'$ are equivalent (up to reverse and mirror image) if and only if the fundamental quandles $Q(K)$ and  $Q(K')$ are isomorphic.  This theorem turns the topological problem of classification of knots into the algebraic problem of classification of quandles.  Recently, there has been many works on classifying some families of quandles, such as connected quandles, Alexander quandles, medial quandles etc. Quandles also have been investigated recently from other point of views such as relations to Lie algebras \cite{CCES1, CCES2},  Hopf algebras \cite{AG, CCES2}, quasigroups and Moufang loops \cite{Elhamdadi}, ring theory \cite{EFT}, Yang-Baxter equation and its homology \cite{CES} and topology\cite{EM}.   Orderability of quandles appeared in \cite{BPS2, DDHPV} where it was shown for example  that  conjugation quandles (respectively core quandles) of bi-orderable groups are right orderable (respectively left orderable) quandles. The set of left ordering on a group $G$ can be given a topology, and this space is called the space of left orderings.  In \cite{Sikora} it was shown that the space of left orderings of a group is compact.  The analogous of this theorem was proved for magmas in general (and quandles in particular) in \cite{DDHPV }.  %{ \color{red} Idrissa: PLEASE write couple phrases here.    Circular orderability of groups was defined in \cite{Cal} where......}

Recently, the notions of circular orderability and left-orderability have
been intensively studied in low dimensional topology, particularly for the fundamental group of 3-manifolds (see \cite{[BaC1}, \cite{[BaC2}, \cite{BS},  \cite{BRW} and \cite{Cal}). In \cite{[BaC1} the authors have shown that a compact, connected, $P^2$-irreducible $3$-manifold has a left circularly orderable fundamental group if and only if there exists a finite cyclic cover with left-orderable fundamental group. A consequence of this result is that, if a $3$-manifold admits a cooriented taut foliation then it has a finite cyclic cover with left-orderable fundamental group, and this give a characterization of a topological property of a $3$-manifold (taut foliation) by an algebraic property to its fundamental group (left-orderability and circular orderability). This may also be a motivation to study some algebraic properties of quandles (circular orderability and orderability).

In this article we  introduce the notion of circular orderability for quandles. We show that the set all right (respectively left) circular orderings of a quandle is a compact topological spaces. 

Let $Q$ be a quandle, we denote by $RCO(Q)$ (respectively $LCO(Q)$) the set of right circular orderings on $Q$ (respectively the set of left circular orderings on $Q$)

\begin{theorem}
    The space $RCO(Q)$ (respectively $LCO(Q)$) is compact.
\end{theorem}

We also show that the space of right (respectively left) orderings \cite{DDHPV} of a quandle embeds in its space of right (respectively left) circular orderings.  

\begin{theorem} If $Q$ is a quandle, then the space of right (respectively left) orderings  of $Q$ is a subspace of its space of right (respectively left) circular orderings.
\end{theorem}

Given a group with a circular ordering which is both right and left invariant, we prove that the conjugation quandle is right circularly orderable and more strongly we shown that the conjugation quandle is right orderable. We also give examples of quandles that are not left circularly orderable and examples of quandles that are neither left nor right circularly orderable.

The following is the organization of this article.  In section~\ref{Review}, we recall the basics of quandles and also orderability.  Section~\ref{Circular} introduces the notion of circular orderability of quandles, and show that a right (respectively left) orderable quandle is right (respectively left) circularly orderable. In this section,
it is also shown that  the conjugation quandle of a bi-orderable group is a right orderable quandle and we give examples of non-left circularly orderable quandles, and explicit examples of quandles of small cardinalities that are neither left nor right circularly orderable.  Section~\ref{Space} deals with the topology of the set of right (respectively left) circular orderings on quandles.  Precisely, we introduce a topology on the set of right (respectively left) circular orderings on quandles and prove that it is compact.

\section{Review of quandles and orderability}\label{Review}

A {\it quandle} is a non-empty set $Q$ with a binary operation $*$ which satisfies the following conditions:
\begin{enumerate}
	\item For any $s\in Q$, $s*s=s$;
	\item For any $s_1, s_2\in Q$, there exists a unique $s_3\in Q$ such that $s_3*s_2=s_1$; and
	\item For any $s_1, s_2, s_3\in Q$, 
	$(s_1*s_2)*s_3=(s_1*s_3)*(s_2*s_3)$.
\end{enumerate}
Let consider right multiplications $R_s: Q \rightarrow Q$ given by $t \mapsto t * s$.  Then axioms (2) and (3) state that $R_s$ are automorphisms of the quandle $Q$ and axiom (1) states that $s$ is a fixed point of $R_s$.
An element $e$ of a quandle $Q$ is called {\it stabilizer element} if $s*e=s$ for any $s\in Q$. 

A quandle $Q$ is called {\it trivial} if all its elements are stabilizer elements.
A trivial quandle can have an arbitrary number of element, this is one difference between a group and a quandle.

By the conditions $(2)$ and $(3)$ of the definition of quandle for any $b\in Q$, there exists an automorphism $R_b:Q\longrightarrow Q$ defined by $R_b(s)=s*b$, for any $s\in Q$. We call $id_Q$ the identity automorphism of $Q$. A quandle is called {\it involutary} if $R_s\circ R_s=id_Q$ for any $s\in Q$.

By the conditions $(2)$ and $(3)$ of the definition of quandle, there exists a dual binary operation $*^{-1}$ defined by $s*^{-1}r=t$ if $s=t*r$ for any $r, s$ and $t\in Q$. Hence, we have also that $(s*r)*^{-1}r=s$ and $(s*^{-1}r)*r=s$ for any $r, s\in Q$, this is called right cancellation. Therefore, by the condition $(1)$ of the definition of quandle and right cancellation we have that $s*^{-1}s=s$ for any $s\in Q$.

A quandle is called {\it latin} if for any $s\in Q$ the map $L_s:Q\longrightarrow Q$ defined by $L_s(t)=s*t$ is a bijection for any $t\in Q$. A quandle is called {\it semi-latin} if the map $L_s$ is injective for any $s\in Q$.

A subquandle of a quandle, $Q$, is a subset $X\subset Q$ which is closed under the quandle operation.

The following are some examples of quandle. 
\begin{itemize}
	\item Any set with the binary operation $x*y=x$ is a quandle called \emph{trivial} quandle.
	\item Let $G$ be a group. The set $S=G$ with the binary operation $*$ defined by $g*h=h^{-1}gh$ is a quandle, denoted Conj$(G)$, and called \emph{conjugation} quandle of $G$.
	\item The set $T=G$, where $G$ is a group, with the binary operation $*$ defined by $g*h=hg^{-1}h$ is a quandle, denoted by Core$(G)$, and called \emph{core} quandle of $G$.  If furthermore $G$ is an abelian group then this quandle is called \emph{Takasaki } quandle \cite{Takasaki}. 
	\item If $\phi$ is an automorphism of a group $G$, then the set $R=G$ with the binary operation $*$ defined by $g*h=\phi(gh^{-1})h$ is a quandle, denoted by Aut$(G, \phi)$, and called \emph{generalized Alexander} quandle of $G$ with respect to $\phi$.
\end{itemize}

Recall that a group $G$ is called \emph{left-orderable} (respectively right-orderable) if there exists a strict total ordering $<$ on $G$ such that $g<h$ implies $fg<fh$ (respectively $gf<hf$) for all $f, g, h \in G$, the relation $<$ is called a \emph{left-ordering} of $G$.  A left-orderable (respectively right-orderable) quandle was defined in the same fashion in \cite{BPS}. A quandle $(Q,*)$ is called left-orderable (respectively right-orderable) if there exists a strict total ordering $\prec$ on $Q$ such that $s\prec t$ implies $r*s\prec r*t$ (respectively $s*r\prec t*r$) for all $r, s, t \in Q$, the relation $\prec$ is called a \emph{left-ordering} (respectively right-ordering) on $Q$.
The notion of a homomorphism between two ordered quandles is natural.  Let $(X,*, \prec_X )$ and $(Y,\diamond, \prec_Y)$ be two right ordered quandles.  A map $f:X \rightarrow Y$ is called an order preserving homomorphism of right ordered quandles if $f$ is a quandle homomorphism (that is,$f(x*y)=f(x) \diamond f(y), \forall x,y \in X$) and $ x \prec_X y$ implies $f(x) \prec_Y f(y)$.

\begin{example}
	We use quandle structures defined in \cite{CES} over the real line $\mathbb{R}$ to give explicit structures of right and left orderable quandles. 	Consider the real line $\mathbb{R}$ with its natural ordering and with quandle operation $x*y=\alpha x +(1-\alpha)y$, where $\alpha \neq 0,1$
	\begin{enumerate}
		\item 
		 If $\alpha >0$, then the quandle $\mathbb{R} $ is right orderable.
		 \item
		 If $\alpha <1$, then the quandle $\mathbb{R} $ is left orderable.
		 \item
		 Thus, if $0< \alpha <1$, then the quandle $\mathbb{R} $ is bi-orderable.

	\end{enumerate}

\end{example}

\section{circular orderability of quandles}\label{Circular}

A {\it circular ordering} on a set $T$ is a map $c: T\times T\times T \rightarrow \{ -1, 0, 1\}$ satisfying:
\begin{enumerate}
	\item If $(t_1, t_2, t_3) \in T^3$ then $c(t_1, t_2, t_3) = 0$ if and only if $\{t_1, t_2, t_3\}$ are not all distinct;
	\item For all $t_1, t_2, t_3, t_4 \in T$ we have
	\[ c(t_1, t_2, t_3) - c(t_1, t_2, t_4) + c(t_1, t_3, t_4)-c(t_2, t_3, t_4) = 0.
	\] 
\end{enumerate}
A set $T$ with a circular ordering is called {\it circularly orderable}. A group $G$ is called {\it left-circularly orderable} (respectively {\it right-circularly orderable}) if it admits a circular ordering $c: G\times G\times G \rightarrow \{ - 1, 0, 1\}$ such that $c(g_1,g_2, g_3)=c(gg_1,gg_2,gg_3)$ (respectively $c(g_1,g_2, g_3)=c(g_1g,g_2g,g_3g)$) for any $g, g_1, g_2$ and $g_3\in G$. In this case we say that the circular ordering $c$ is left-invariant (respectively right-invariant). 

We say that a quandle $Q$ is {\it left-circularly orderable} (respectively {\it right-circularly orderable}) if it admits a circular ordering $c: Q\times Q\times Q \rightarrow \{ - 1, 0, 1\}$ such that $c(s_1,s_2, s_3)=c(s*s_1,s*s_2,s*s_3)$ (respectively $c(s_1,s_2, s_3)=c(s_1*s,s_2*s,s_3*s)$) for any $s, s_1, s_2$ and $s_3\in Q$. 

\begin{lemma}\label{ordering}
	A left (respectively right) orderable quandle is also left (respectively right) circularly orderable.
\end{lemma}
\begin{proof}
	Let $\prec$ be a left ordering on a quandle $Q$ then define the map $c_{\prec}:Q^3 \rightarrow \{ - 1, 0, 1\}$ by 
	$$(1)\;\;\;\; c_{\prec}(s_1,s_2,s_3)=\left\{\begin{array}{cl} 1 & \text{if $s_1\prec s_2\prec s_3$ or $s_2\prec s_3\prec s_1$ or $s_3\prec s_1\prec s_2$}
		\\ -1  &  \text{if $s_1\prec s_3\prec s_2$ or $s_3\prec s_2\prec s_1$ or $s_2\prec s_1\prec s_3$  }  
		\\ 0 & \text{otherwise.}
	\end{array}
	\right.$$ 
	
	By definition, $c_{\prec}$ is a circular ordering. Now, it is left to check that $c_{\prec}$ is invariant under left-multiplication. Since $Q$ is left orderable, multiplying on the left by any element $s\in Q$ will not change the inequalities of $(1).$ Similarly, if a quandle $Q$ is right orderable then the same proof show that $Q$ is right circularly orderable.
\end{proof}

\begin{proposition}
	Let $G$ be a circularly orderable group with a circular ordering $c$ which is left and right-invariant. Then, the quandle $\mathrm{Conj}(G)$ is a right circularly orderable quandle;

\end{proposition}
\begin{proof}
	Let $d: \mathrm{Conj}(G)\times \mathrm{Conj}(G)\times \mathrm{Conj}(G) \rightarrow \{ -1, 0, 1\}$ defined by $d(g_1, g_2, g_3)=c(g_1, g_2, g_3)$ for any $g_1, g_2$ and $g_3\in G$. Let $g_1, g_2$ and $g_3\in G$, since $d(g_1, g_2, g_3)=c(g_1, g_2, g_3)$, $d$ is a circular ordering on 
		$\mathrm{Conj}(G)$ as a set. Now it is left to show the right-invariant condition. Let $g, g_1, g_2$ and $g_3\in G$,
		we have that $d(g_1*g, g_2*g, g_3*g)=d(g^{-1}g_1g, g^{-1}g_2g, g^{-1}g_3g)=c(g^{-1}g_1g, g^{-1}g_2g, g^{-1}g_3g)=c(g_1, g_2, g_3)$. Hence, $\mathrm{Conj}(G)$ is a right circularly orderable quandle.

\end{proof}

Recall that a group $G$ is called bi-circularly orderable (respectively bi-orderable) if it admits a circular ordering (respectively an ordering) which is left and right-invariant. Similarly, we say that a quandle $Q$ is bi-circularly orderable  if it admits a circular ordering (respectively an ordering) which is left and right-invariant. 
The notion of bi-circular orderability for groups have been studied since the $1950$'s, and in $1959$ Swierczkowski had given in \cite{Sw} a complete classification of these type of groups.

\begin{theorem}{\cite{Sw}}\label{bio}
    If a group $G$ is bi-circularly orderable, then there exists a bi-orderable group $\Gamma$ such that $G$ is a subgroup of $\Gamma\times \mathbb{S}^1$. Moreover, the bi-circular order on $G$ is determined by the natural bi-circular order on $\Gamma\times \mathbb{S}^1$.
\end{theorem}

By the natural bi-circular order it is meant that if you equip $\Gamma \times \mathbb{S}^1$ with the lexicographic bi-circular ordering, then the bi-circular ordering on $G$ is just the restriction of this lexicographic circular ordering on $\Gamma \times \mathbb{S}^1$ to the subgroup $G$. The Goursat's Lemma Give a classification of all the subgroups of $\Gamma \times \mathbb{S}^1$. Hence, Theorem \ref{bio} gives infinitely many example of bi-circularly orderable groups that are not bi-orderable.

Let $\{Q_i, *_i\}_{\i\in I}$ be a family of quandles. Then the Cartesian product $Q=\prod_{i \in I} Q_i$ with the operation $(x_i)\star (y_i)=(x_i*_iy_i)$ is a quandle called the {\it quandle product}.

\begin{proposition}
The conjugation quandle of any bi-circularly orderable group $G$
 is right orderable.
 \end{proposition}
\begin{proof}
Let $H$ be any bi-orderable group. Since $\mathrm{Conj}(H\times \mathbb{S}^1)=\mathrm{Conj}(H)\times \mathrm{Conj}(\mathbb{S}^1)$, and $\mathbb{S}^1$ is commutative, so $\mathrm{Conj}(\mathbb{S}^1)$ is right orderable as trivial quandle. Since $H$ is bi-orderable, $\mathrm{Conj}(H)$ is right orderable. Hence, the quandle $\mathrm{Conj}(H\times \mathbb{S}^1)=\mathrm{Conj}(H)\times \mathrm{Conj}(\mathbb{S}^1)$ is right orderable by \cite[Proposition 4.4 ]{RSS}. Therefore, the result follows from Theorem \ref{bio}. 
\end{proof}

\begin{lemma}
    If $G$ is a nontrivial group with at least three elements, then the quandle $\mathrm{Conj}(G)$ is not left circularly orderable.
\end{lemma}
\begin{proof}
Let $G$ be any nontrivial group with at least three elements. Consider the conjugation quandle $\mathrm{Conj}(G)$ with operation $x*y=y^{-1}xy$. By contradiction, let $c:\mathrm{Conj}(G)^3 \rightarrow \{ \pm 1, 0\}$ be a left circular ordering.  Then $c(s_1,s_2,s_3)=c(e*s_1,e*s_2, e*s_3)=c(e,e,e)=0$, where $e$ is the identity element of $G$.  Thus the function $c$ is the constant zero map implying that $\mathrm{Conj}(G)$ cannot be have a left circular ordering.
\end{proof}

\begin{example}
	The trivial quandle with \emph{two} elements is bi-circularly orderable since the trivial circular ordering which is always equal to zero is both left and right invariant. Since this quandle is not left orderable, this show that the notion of circular orderability for quandles is different to the notion orderability.
	
\end{example}

\begin{example}
	Consider the three element quandle $X=\{1,2,3\}$ with orbit decomposition $\{1,2\} \sqcup \{3\}$.  It's Cayley table is given by
	$\begin{bmatrix} 
		1 & 1 & 2 \\
		2 & 2 & 1 \\
		3& 3 & 3 
	\end{bmatrix}.$  The $(i,j)$-entry of this matrix correspond to the element $i*j$ in the quandle. We show that this quandle is neither left nor right circularly orderable.\\
\noindent
  (I) Now assume that $c$ is left-invariant circular ordering on $X$, then for any pairwise distinct elements $x,y,z \in X$, we have $c(x,y,z)=c(3*x,3*y,3*z)=c(3,3,3)=0,$ and thus the quandle $X$ can't have a left-invariant circular ordering.\\
	(II) Now assume that $c$ is right-invariant circular ordering on $X$, then we have $c(1,2,3)=c(1*3,2*3,3*3)=c(2, 1, 3)$. Since $c$ is a circular ordering, it satisfies condition $(2)$ of the definition, so
	\[ c(1, 2, 1) - c(1, 2, 3) + c(1, 1, 3)-c(2, 1, 3) = 0.
	\]
Hence, $c(1,2,3)=-c(2,1,3)$	 which is a contradiction to the fact that $$c(1,2,3)=c(1*3,2*3,3*3)=c(2, 1, 3).$$

\end{example}

\begin{example}
	Consider the dihedral quandle $X=\mathbb{Z}_3$, where $x*y=2y-x=2y+2x$. We show that this quandle is neither left nor right circularly orderable. The quandle operation satisfies $x*y=y*x$ for all $x,y \in 	X$, and then left circular orderings and right circular orderings are the same. By contradiction assume that there exists a right circular ordering $c$ on $X$, and let $x,y,z \in  X$.  If at least two of $x,y,z$ are equal then by definition $c(x,y,z)=0$.  Now assume that $x,y,z$ are pairwise distinct. Thus the product of any two of them gives the third element.  Then $c(x,y,z)=c(x*x,y*x,z*x)=c(x,z,y)$, $c(x,y,z)=c(x*y,y*y,z*y)=c(z,y,x)$ and $c(x,y,z)=c(x*z,y*z,z*z)=c(y,x,z)$.  We thus obtain $c(x,y,z)=c(x,z,y)=c(z,y,x)=c(y,x,z)$ and consequently  any of the 6 permutations of the set $\{x,y,z\}$ leaves the value of $c(x,y,z)$ unchanged.
Since $c$ is a circular ordering, it satisfies condition $(2)$ of the definition, so
	\[ c(x, y, x) - c(x, y, z) + c(x, x, z)-c(y, x, z) = 0.
	\]
Hence, $c(x,y,z)=-c(y,x,z)$	 which is a contradiction to the fact that	
$$c(x,y,z)=c(y,x,z).$$ Therefore, the dihedral quandle $X=\mathbb{Z}_3$ with quandle operation $x*y=2y-x=2y+2x$ is neither left nor right circular orderable.

\end{example}

%\end{proof}

\section{The space of circular orderings of quandles}\label{Space}
In this section we study the set of circular orderings of a quandle. Let $Q$ be a quandle, we denote by $RCO(Q)$ (respectively $LCO(Q)$) the set of right circular orderings on $Q$ (respectively the set of left circular orderings on $Q$).
Since a right ordering (respectively left ordering) on a quandle gives a right circular ordering (respectively left circular ordering) by Lemma \ref{ordering}, the set of right orderings $RO(Q)$ (respectively left orderings $LO(Q)$) on a quandle $Q$ can be seen as a subset of the set of left circular orderings $LCO(Q)$. The set of left-orderings (respectively right orderings) on a quandle was already studied in \cite{DDHPV} and they showed that for any magma $\mathcal{M}$, the set of left orderings $LO(M)$ (respectively right orderings $RO(M)$) is a compact topological space. A magma is a set with a binary operation $\cdot :\mathcal{M}\times\mathcal{M}\longrightarrow\mathcal{M}.$ So, a quandle is in fact a magma. 

We topologize the set $RCO(Q)$ (respectively $LCO(Q)$) as a subspace of the space $\{-1, 0, 1\}^{Q^3}$ of all maps from $Q^3$ to $\{-1, 0, 1\}$ with the Tychonoff topology. We define $$\Gamma(Q):=\{S=(x, y, z)\in Q^3\; |\; x=y\; {\rm or} \; y=z\; {\rm or}\; x=z\},$$ $$R_S=\{c\in RCO(Q)\;|\; c(S)=1\}$$ and $$L_S=\{c\in LCO(Q)\;|\; c(S)=1\}$$ for any $S\in Q^3\setminus \Gamma(Q)$.

\begin{lemma}
    The set $\{R_S\}_{S\in Q^3\setminus\Gamma(Q)}$ (respectively $\{L_S\}_{S\in Q^3\setminus\Gamma(Q)}$) is a subbasis for the topology on $RCO(Q)$ (respectively $LCO(Q)$).
\end{lemma}
\begin{proof}
    Since the $\{-1, 0, 1\}^{Q^3}$ is a Cantor set with a subbasis 
    $\{V_{s, -1}, V_{S, 0}\; {\rm and }\; V_{S, 1}\}_{S\in Q^3}$ Where $$V_{S, -1}=\{f:Q^3\rightarrow \{-1, 0, 1\}\; |\; f(S)=-1\},$$
    $$V_{S, 0}=\{f:Q^3\rightarrow \{-1, 0, 1\}\; |\; f(S)=0\}$$
    and $$V_{S, 1}=\{f:Q^3\rightarrow \{-1, 0, 1\}\; |\; f(S)=1\}.$$
    
    Since $RCO(Q)\cap V_{S, 0}=RCO(Q)$ (respectively $LCO(Q)\cap V_{S, 0}=LCO(Q)$) for any $S\in \Gamma(Q)$ and $RCO(Q)\cap V_{S, 0}=\emptyset$ (respectively $LCO(Q)\cap V_{S, 0}=\emptyset$) for any $S\in Q^3\setminus\Gamma(Q)$, we can throw out all sets of the form $RCO(Q)\cap V_{S, 0}$.
    
    Since $RCO(Q)\cap V_{S, -1}=RCO(Q)\cap V_{\tau.S, 1}$ (respectively $LCO(Q)\cap V_{s, -1}=LCO(Q)\cap V_{\tau.S, 1}$) for any transposition $\tau$, we can also throw out all sets of the form $RCO(Q)\cap V_{S, -1}$ (respectively $LCO(Q)\cap V_{\tau.S, 1}$).
    
\end{proof}
%\vspace{1cm}
\begin{theorem}\label{compact}
    The space $RCO(Q)$ (respectively $LCO(Q)$) is compact.
\end{theorem}

\begin{proof}

We have that \begin{multline*}RCO(Q)=\{c: Q^3 \rightarrow \{ -1, 0, 1\} \;|\; {\rm if}\; (t_1, t_2, t_3) \in Q^3\; {\rm then}\; c(t_1, t_2, t_3) = 0\\ {\rm if\; and\; only\; if}\; (t_1, t_2, t_3)\in \Gamma(Q);\\ {\rm for\; all}\; (t_1, t_2, t_3, t_4) \in Q^4\;\\ {\rm we\; have}\; c(t_1, t_2, t_3) - c(t_1, t_2, t_4) + c(t_1, t_3, t_4)-c(t_2, t_3, t_4) = 0;\\ {\rm and }\; c(t_1*t, t_2*t, t_3*t)=c(t_1,t_2,t_3)\; {\rm for\; any}\; (t, t_1,t_2,t_3)\in Q^4\}\end{multline*}

Let 

\begin{align*}
    A:&=\{c: Q^3 \rightarrow \{ -1, 0, 1\} \;|\; {\rm if}\; S=(t_1, t_2, t_3) \in Q^3\; {\rm then}\; c(t_1, t_2, t_3) = 0\; {\rm if\; and\; only\; if}\;\\ &S\in \Gamma(Q)\}\\
    &= (\bigcap_{S\in \Gamma(Q)}V_{S, 0})\cap(\bigcap_{S\in Q^3\setminus\Gamma(Q)}(V_{S,-1}\cup V_{S,1})),\\
\end{align*}

\begin{align*}
    B:&=\{c: Q^3 \rightarrow \{ -1, 0, 1\} \;|\;{\rm for\; all}\; W=(t_1, t_2, t_3, t_4) \in Q^4\; {\rm we\; have}\\ &d_W(c)=c(t_1, t_2, t_3) - c(t_1, t_2, t_4) + c(t_1, t_3, t_4)-c(t_2, t_3, t_4) = 0\}\\
    &=\bigcap_{W\in Q^4}\{c\in \{-1, 0, 1\}^{Q^3}\;|\;  d_W(c)=0\}
\end{align*}

where for any $W=(t_1, t_2, t_3, t_4)\in Q^4$ we define $d_W: \{-1, 0, 1\}^{Q^3}\longrightarrow \mathbb{Z}$ by 
$$d_W(c)=c(t_1, t_2, t_3) - c(t_1, t_2, t_4) + c(t_1, t_3, t_4)-c(t_2, t_3, t_4)$$
and 
\begin{align*}
    C:&=\{c: Q^3 \rightarrow \{ -1, 0, 1\} \;|\; c(t_1*t, t_2*t, t_3.t)=c(t_1,t_2,t_3)\; {\rm for\; any}\; (t, t_1,t_2,t_3)\in Q^4\}\\
    &=\bigcap_{t\in Q, (t_1, t_2, t_3)\in Q^3}\{c\in \{-1, 0, 1\}^{Q^3}\;|\; c(t_1*t, t_2*t, t_3*t)=c(t_1,t_2,t_3)\}.
\end{align*}
Since $RCO(Q)=A\cap B\cap C$ and the subsets $V_{S, 0}$, $V_{S,-1}$, $V_{S,1}$, $\{c\in \{-1, 0, 1\}^{Q^3}\;|\; d_W(c)=0\}$ and $\{c\in \{-1, 0, 1\}^{Q^3}\;|\; c(t_1*t, t_2*t, t_3*t)=c(t_1,t_2,t_3)\}$ are closed, $RCO(Q)$ is closed in $ \{-1, 0, 1\}^{Q^3}$. Therefore, $RCO(Q)$ is compact.

A similar proof show that $LCO(Q)$ is compact.
\end{proof}
Given a quandle $Q$, recall that a topology on $LO(Q)$ is constructed in \cite{DDHPV} by choosing as a subbasis the collection $\{V_{(a, b)}\}_{(a,b)\in Q\times Q\setminus \Delta}$, where $\Delta=\{(a, a)\in Q\;|\; a\in Q\}$, and $V_{(a,b)}=\{<\in LO(Q)\;|\; a<b\}$. Similarly, a topology on $RO(Q)$ is constructed by choosing as a subbasis the collection $\{V_{(a, b)}\}_{(a,b)\in Q\times Q\setminus \Delta}$, where $\Delta=\{(a, a)\in Q\;|\; a\in Q\}$ and  $V_{(a,b)}=\{<\in RO(Q)\;|\; a<b\}$.

\begin{proposition}
\label{embedding}
The inclusion map $i:LO(Q) \rightarrow LCO(Q)$ (respectively $j:RO(Q) \rightarrow RCO(Q)$) given by 
\[ < \mapsto c_<,
\]
is an embedding.
\end{proposition}

\begin{proof}
Recall that the collection $\{L_S\}_{S\in Q^3\setminus\Gamma(Q)}$ (respectively $\{R_S\}_{S\in Q^3\setminus\Gamma(Q)}$) is a subbasis for the topology on $LCO(Q)$ (respectively $RCO(Q)$), where $$R_S=\{c\in RCO(Q)\;|\; c(S)=1\}$$ and $$L_S=\{c\in LCO(Q)\;|\; c(S)=1\}$$ for any $S=(x, y, z)\in Q^3\setminus \Gamma(Q)$.

For any $S=(x, y, z)\in Q^3\setminus \Gamma(Q)$, we have that
\begin{align*}
    i^{-1}(L_S)&=\{<\in LO(Q)\;|\; c_<(S)=1\}\\
    &=\{<\in LO(Q)\;|\; x<y<z\; {\rm or}\; y<z<x\; {\rm or} \; z<x<y\}\\
    &=(V_{(x,y)}\cap V_{(y,z)})\bigcup (V_{(y,z)}\cap V_{(z,x)})\bigcup (V_{(z,x)}\cap V_{(x,y)})
\end{align*}
and 
\begin{align*}
    i^{-1}(R_S)&=\{<\in RO(Q)\;|\; c_<(S)=1\}\\
    &=\{<\in RO(Q)\;|\; x<y<z\; {\rm or}\; y<z<x\; {\rm or} \; z<x<y\}\\
    &=(V_{(x,y)}\cap V_{(y,z)})\bigcup (V_{(y,z)}\cap V_{(z,x)})\bigcup (V_{(z,x)}\cap V_{(x,y)}).
\end{align*}
Therefore, both of the injective maps $i$ and $j$ are continuous. Since both $LO(Q)$ and $RO(Q)$ are compact and $\{-1, 0, 1\}^{Q^3}$ is Hausdorff, the maps $i$ and $j$ are embeddings.
\end{proof}
\begin{corollary}
The inclusion map $k:BO(Q)= LO(Q)\cap RO(Q)\longrightarrow LCO(Q)$ (respectively $k':BO(Q)= LO(Q)\cap RO(Q)\longrightarrow RCO(Q)$) is an embedding.
\end{corollary}

We end this section by asking few questions about the notion of circular orderability for quandles.

\begin{question}
For exactly which quandles $Q$, the space $RCO(Q)$ (respectively $LCO(Q)$) is finite?
\end{question}
\begin{question}
Is there a characterization of the notion of right circular orderability (respectively left circular orderability) of a quandle in term of quandles actions?
\end{question}

\begin{question}
For exactly which quandles $Q$, the embeddings $i:LO(Q) \rightarrow LCO(Q)$ and $j:RO(Q) \rightarrow RCO(Q)$ are homeomorphisms?
\end{question}

\end{document}